\documentclass[11pt]{amsart}

\usepackage[T1]{fontenc}
\usepackage[utf8]{inputenc}
\usepackage{lmodern}
\usepackage{amsmath,amssymb,amsthm,mathtools}
\usepackage{microtype}
\usepackage{needspace}
\usepackage{tikz}
\usepackage[hidelinks]{hyperref}

\newtheorem{theorem}{Theorem}
\newtheorem{lemma}[theorem]{Lemma}
\newtheorem{corollary}[theorem]{Corollary}

\title[Meneguette's Polynomial Problem]
{A Solution to Meneguette's Polynomial Problem}

\author{K. Castillo}

\address{CMUC, University of Coimbra, 3000-143 \mbox{Coimbra}, Portugal}

\email{math@keniercastillo.com}

\date{26 July 2026}

\subjclass[2020]{30C15, 26C10, 65L20}

\keywords{Polynomial zeros, coefficient perturbation, unit disc,
Brown methods, zero stability}

\hypersetup{
  pdftitle={A Solution to Meneguette's Polynomial Problem},
  pdfauthor={K. Castillo},
  pdfsubject={Polynomial zero location and stability of Brown methods},
  pdfkeywords={polynomial zeros, coefficient perturbation, unit disc,
    Brown methods, zero stability}
}

\begin{document}

\begin{abstract}
In 1979, Jeltsch [Numer. Math. 32 (1979), 167--181]
conjectured that every zero-stable Brown method is stiffly stable. For
Brown $(K,L)$ methods, Meneguette subsequently reduced the relevant
$A_0$-stability question to a zero-location property for a family of
perturbed characteristic polynomials. A closely related perturbation
question already appeared in his 1987 Oxford doctoral thesis, and he
later posed a broader purely polynomial version of the problem in 1994
[SIAM Review 36 (1994), 656--657]. We provide a complete solution to Meneguette’s polynomial problem.
\end{abstract}

\maketitle

\section{Introduction and main result}

The zero-location problem studied here arose from the stability
analysis of Brown's implicit multistep multi-derivative methods for
stiff ordinary differential equations~\cite{Brown77}. In 1979, Jeltsch
conjectured that every zero-stable Brown method is stiffly
stable~\cite{J79}. Meneguette subsequently showed that the relevant
$A_0$-stability question for Brown $(K,L)$ methods could be reduced
to determining the location of the zeros of a one-parameter family of
perturbed characteristic polynomials~\cite[Theorem~2.13, Remark~2.16 and
pp.~105--107]{M87}. He also obtained the coefficient monotonicity and
strong-stability results used in the Brown-method application
below~\cite[Theorems~2.9 and~2.12, pp.~66--72]{M87}.

He later isolated the underlying zero-location question as a problem
about arbitrary polynomials. It appeared as Problem~94-19 in
1994~\cite{M94}. Let
$$
P(z)=a_0+a_1z+\cdots+a_{n-1}z^{n-1}+a_nz^n,
$$
where
$$
0<a_0\leq a_1\leq\cdots\leq a_{n-1},
\quad
a_n>0.
$$
Suppose that all the zeros of $P$ lie in the unit disc and that
$$
na_n>(n-1)a_{n-1}.
$$
Must all the zeros of
$$
P_d(z):=P(z)+dz^n
$$
also lie in the unit disc for every $d>0$?

The original formulation leaves open whether the unit disc is to be
understood in the open or the closed sense. Theorem~\ref{thm:main}, stated below,
resolves this ambiguity: the original polynomial may have zeros on the
unit circle, whereas every positive perturbation has all its zeros in
the open unit disc.

For $n>1$, Blondel recorded a partial answer in
1995~\cite{Blondel95}, observing from the Enestr\"om--Kakeya theorem
that the conclusion holds whenever
$$
d\geq\frac{a_n}{n-1}.
$$
The hypothesis of Problem~94-19 gives
$$
a_{n-1}-a_n<\frac{a_n}{n-1}.
$$
A direct application of the same theorem therefore gives the sharper
closed-disc conclusion whenever $d\geq a_{n-1}-a_n$, because the
coefficients of $P_d$ are then non-decreasing. If
$d>a_{n-1}-a_n$, the final coefficient inequality is strict; the
equality characterisation in the standard proof of the
Enestr\"om--Kakeya theorem would then force a unit-circle zero $w$ to
satisfy $w^n=w^{n+1}=1$. Hence $w=1$, which is impossible because
$P_d(1)>0$, so all the zeros lie in $\mathbb D$. At the endpoint
$d=a_{n-1}-a_n$, however, this elementary argument yields only
$\overline{\mathbb D}$. Consequently, if $a_n\geq a_{n-1}$, the
open-disc conclusion follows immediately for every $d>0$. If
$a_n<a_{n-1}$, what remains beyond the elementary argument is
$$
\frac{n-1}{n}a_{n-1}<a_n<a_{n-1},
\quad
0<d\leq a_{n-1}-a_n
$$
for the open-disc conclusion (and the same range with a strict upper
inequality for the closed-disc interpretation of the original
problem). Thus the substance of Problem~94-19 lies in the
small-perturbation regime, including the endpoint when the strict
conclusion is sought.

The problem was subsequently investigated under additional structural
assumptions or coefficient conditions. In 2008, Meneguette, Botta and
Cuminato treated the reflexive case~\cite[Proposition~3.1]{MBC08}.
Under the coefficient monotonicity in Problem~94-19, reflexivity forces
the polynomial, for $n>1$, to have the form
$$
P(z)
=
a_n(1+z^n)
+
a_{n-1}(z+z^2+\cdots+z^{n-1}).
$$
In the non-trivial case $a_n<a_{n-1}$, they obtained the sufficient
perturbation condition
$$
d\geq a_{n-1}-2a_n.
$$
Since $na_n>(n-1)a_{n-1}$ implies $2a_n>a_{n-1}$ when $n>1$,
this establishes the closed-disc conclusion for every $d>0$; the
boundary case $a_n=a_{n-1}$ already follows directly from the
Enestr\"om--Kakeya theorem. Thus the prior work settles the reflexive
case in the closed-disc sense. Theorem~\ref{thm:main} below also gives
the strict open-disc conclusion.

Botta's doctoral thesis proves the same closed-disc reflexive result,
together with the following additional sufficient
case~\cite[Propositions~2.2 and~2.3]{Botta08}: for some integer $p>1$,
$$
\begin{aligned}
0<a_0<a_1<\cdots<a_{n-1},
\quad&
0<a_n<a_{n-1},
\\[7pt]
p a_j<(p-1)a_{j+1},
\quad&
j=0,\ldots,n-2,
\quad
p a_n>(p-1)a_{n-1}.
\end{aligned}
$$
Botta and Meneguette subsequently pursued both analytical and numerical
investigations of leading-coefficient perturbations, including a search
for counterexamples~\cite{BM10}.

In 2012, Botta et al.\ formulated Meneguette's conjecture with the
additional hypothesis $a_n<a_{n-1}$, and proved it under stronger
coefficient-ratio conditions~\cite{BMCM12}. More precisely, they
required the existence of $r\in(0,1)$ such that
$$
\begin{aligned}
0<a_0<a_1<\cdots<a_{n-1},
\quad&
0<a_n<a_{n-1},
\\[7pt]
a_j<r a_{j+1},
\quad&
j=0,\ldots,n-2,
\quad
r a_{n-1}<a_n.
\end{aligned}
$$
Under these hypotheses, they proved that, for every $\gamma>0$, all
the zeros of $P(z)+\gamma z^n$ lie in the open unit disc. They also
applied this result to a subclass of Brown methods.

After division by $a_n$, the reflexive polynomial considered
in~\cite{MBC08} has the normalised palindromic form
$$
R_\lambda(z)
=
1+\lambda(z+z^2+\cdots+z^{n-1})+z^n,
\quad
\lambda=\frac{a_{n-1}}{a_n}.
$$
For $n>1$, Botta, Marques and Meneguette completely characterised in
2014 the values of $\lambda$ for which the zeros of $R_\lambda$ lie on
the unit circle~\cite[Theorem~3.1]{BMM14}. They also proved that, for
$\lambda\geq0$ and $\gamma>0$, all the zeros of
$R_\lambda(z)+\gamma z^n$ lie in the closed unit disc precisely when
$\gamma\geq\lambda-2$~\cite[Theorem~3.3]{BMM14}. This sharp criterion
extends and refines the reflexive analysis of~\cite{MBC08}. In the
parameter range of Problem~94-19 it applies to every positive
perturbation.

None of the preceding results addresses the general
small-perturbation regime under the hypotheses of Problem~94-19 alone.
Theorem~\ref{thm:main} below resolves this case without imposing an auxiliary
ratio condition, reflexivity or palindromicity, or a positive lower
bound on the perturbation. It permits zeros of the unperturbed
polynomial on the unit circle, yet places every zero of each positively
perturbed polynomial in the open unit disc.

The proof uses only the theory of polynomial zeros and is independent
of the stability-theoretic machinery from which the problem arose. In
the final section, the theorem is combined with Meneguette's coefficient
and strong-stability results; the resulting range is then compared with
the earlier coefficient-ratio argument of~\cite{BMCM12}.

Throughout, we write
$$
\mathbb D:=\{z\in\mathbb C:|z|<1\},
\quad
\overline{\mathbb D}:=\{z\in\mathbb C:|z|\leq1\}.
$$

\begin{theorem}\label{thm:main}
Let $n$ be a positive integer and let
$$
P(z)=\sum_{j=0}^n a_jz^j,
$$
where $a_0,\ldots,a_n\in\mathbb R$ satisfy
$$
a_j>0, \quad 0\leq j\leq n,
\quad
a_0\leq a_1\leq\cdots\leq a_{n-1}.
$$
Suppose that every zero of $P$ lies in
$\overline{\mathbb D}$ and that
$$
na_n>(n-1)a_{n-1}.
$$
Then, for every $d>0$, all the zeros of
$P_d(z):=P(z)+dz^n$ lie in $\mathbb D$.
\end{theorem}

\section{Proof of the main result}

\begin{proof}
If $n=1$, the hypothesis
$-a_0/a_1\in\overline{\mathbb D}$ gives $a_0\leq a_1$. The zero of
$P_d$ is $-a_0/(a_1+d)$, whose modulus is strictly less than
$a_0/a_1\leq1$, and hence lies in $\mathbb D$. We may therefore
suppose that $n\geq2$.

\subsubsection*{Reciprocal formulation}
For $t\geq a_n$, consider the one-parameter family
$$
P_t(z):=a_0+a_1z+\cdots+a_{n-1}z^{n-1}+tz^n.
$$
Thus $P_{a_n}=P$ and $P_{a_n+d}=P_d$.

It is convenient to introduce the reciprocal polynomial
$$
Q_t(w):=t+a_{n-1}w+a_{n-2}w^2+\cdots+a_0w^n,
\quad
w\in\mathbb C.
$$
For $w\neq0$, it satisfies
$$
Q_t(w)=w^nP_t\!\left(\frac{1}{w}\right).
$$
Since $a_0>0$, the polynomial $Q_t$ has degree $n$.
For $1\leq j\leq n$, set $b_j:=a_{n-j}$. Then
\begin{equation}\label{eq:Qt}
Q_t(w)=t+\sum_{j=1}^n b_jw^j.
\end{equation}
The monotonicity of the coefficients of $P$ yields
$$
b_1\geq b_2\geq\cdots\geq b_n>0.
$$
In particular, $b_1=a_{n-1}$.

Since $a_0>0$ and $t\geq a_n>0$, neither $P_t$ nor $Q_t$
vanishes at the origin. The reciprocal relation gives a
multiplicity-preserving correspondence between their zeros, under
which $w$ corresponds to $1/w$. It follows that all the zeros of
$P_t$ lie in $\mathbb D$ precisely when every zero of $Q_t$ has
modulus greater than $1$. Likewise, all the zeros of $P_t$ lie in
$\overline{\mathbb D}$ precisely when every zero of $Q_t$ has
modulus at least $1$.

We shall use the following auxiliary lemma, whose proof is deferred to
Section~\ref{sec:trig-lemma}.

\begin{lemma}\label{lem:trig}
Let $n$ be a positive integer, and let
$$
b_1\geq b_2\geq\cdots\geq b_n\geq0,
$$
and let $\theta\in(0,2\pi)$. If
\begin{equation}\label{eq:imagcondition}
\sum_{j=1}^n b_j\sin(j\theta)=0,
\end{equation}
then
\begin{equation}\label{eq:lemmaineq}
(n-1)b_1+
\sum_{j=1}^n(n-j)b_j\cos(j\theta)
\geq0.
\end{equation}
\end{lemma}

\subsubsection*{Zeros on the unit circle}
Fix $t\geq a_n$, and suppose that $Q_t$ has a zero $w$ on the
unit circle. Since $Q_t(1)>0$, this zero is not $1$; we may
therefore write $w=e^{i\theta}$, where $0<\theta<2\pi$. Equating
real and imaginary parts in
$$
t+\sum_{j=1}^n b_je^{ij\theta}=0,
$$
we obtain
\begin{equation}\label{eq:circleimag}
\sum_{j=1}^n b_j\sin(j\theta)=0,
\end{equation}
and
\begin{equation}\label{eq:circlereal}
t+\sum_{j=1}^n b_j\cos(j\theta)=0.
\end{equation}

Set
$$
S:=\sum_{j=1}^n j b_j\cos(j\theta).
$$
Applying Lemma~\ref{lem:trig} to~\eqref{eq:circleimag} gives
$$
(n-1)b_1+
\sum_{j=1}^n(n-j)b_j\cos(j\theta)
\geq0.
$$
On the other hand, equation~\eqref{eq:circlereal} gives
$$
\sum_{j=1}^n(n-j)b_j\cos(j\theta)=-nt-S.
$$
It follows that
$$
S\leq(n-1)b_1-nt.
$$
Since $b_1=a_{n-1}$ and $t\geq a_n$, the hypothesis of the theorem
yields
$$
nt\geq na_n>(n-1)a_{n-1}=(n-1)b_1.
$$
Thus $S<0$.

Since
$$
wQ_t'(w)=\sum_{j=1}^n j b_je^{ij\theta},
$$
we have
$$
\operatorname{Re}\bigl(wQ_t'(w)\bigr)
=
\sum_{j=1}^n j b_j\cos(j\theta)
=
S<0.
$$
In particular, $Q_t'(w)\neq0$, and every zero of $Q_t$ on the unit
circle is simple.

\subsubsection*{Transversal crossings}
Let $t_0\geq a_n$, and suppose that
$$
Q_{t_0}(w_0)=0,
\quad
|w_0|=1.
$$
Although the family was introduced for $t\geq a_n$, the same
expression as in~\eqref{eq:Qt} defines $Q_t$ for every real $t$.
Write $w_0=x_0+iy_0$ and define
$$
F(t,x,y)
:=
\bigl(
\operatorname{Re}Q_t(x+iy),
\operatorname{Im}Q_t(x+iy)
\bigr).
$$
At $(t_0,x_0,y_0)$, the Jacobian with respect to $(x,y)$ is
$$
D_{(x,y)}F(t_0,x_0,y_0)
=
\begin{pmatrix}
\operatorname{Re}Q_{t_0}'(w_0)
&
-\operatorname{Im}Q_{t_0}'(w_0)
\\[7pt]
\operatorname{Im}Q_{t_0}'(w_0)
&
\operatorname{Re}Q_{t_0}'(w_0)
\end{pmatrix}.
$$
Hence
$$
\det D_{(x,y)}F(t_0,x_0,y_0)
=
\bigl|Q_{t_0}'(w_0)\bigr|^2>0,
$$
because $w_0$ is simple. The real implicit function theorem therefore
gives a unique local $C^1$ branch $w(t)$, defined near $t_0$, such
that
$$
w(t_0)=w_0,
\quad
Q_t(w(t))=0.
$$
Since $\partial Q_t/\partial t=1$, differentiation gives
\begin{equation}\label{eq:wprime}
w'(t)=-\frac{1}{Q_t'(w(t))}.
\end{equation}

Set $\rho(t):=|w(t)|^2$.
Then
$$
\rho'(t)
=
2\operatorname{Re}\bigl(\overline{w(t)}\,w'(t)\bigr).
$$
Since $|w_0|=1$, using~\eqref{eq:wprime} we obtain
$$
\begin{aligned}
\rho'(t_0)
&=
-2\operatorname{Re}
\left(
\frac{\overline{w_0}}{Q_{t_0}'(w_0)}
\right)
\\[7pt]
&=
-2\,
\frac{
\operatorname{Re}\bigl(w_0Q_{t_0}'(w_0)\bigr)
}{
\bigl|w_0Q_{t_0}'(w_0)\bigr|^2
}.
\end{aligned}
$$
The numerator is strictly negative, so $\rho'(t_0)>0$. Since
$\rho(t_0)=1$, there is an $\varepsilon>0$ such that
$\rho(t)>1$ for $t_0<t<t_0+\varepsilon$, whereas
$\rho(t)<1$ for $t_0-\varepsilon<t<t_0$. Every unit-circle
contact is therefore transversal: as $t$ increases through $t_0$,
the corresponding zero crosses from the interior to the exterior of
the unit circle.

\subsubsection*{Initial displacement}
We now pass from the local crossing property to the global conclusion.
Since every zero of $P=P_{a_n}$ lies in
$\overline{\mathbb D}$, reciprocity places every zero of
$Q_{a_n}$ on or outside the unit circle. Let
$\xi_1,\ldots,\xi_m$ be its distinct unit-circle zeros, with $m=0$
if there are none. Each $\xi_j$ is simple and therefore lies on a
unique local $C^1$ branch $\xi_j(t)$ satisfying
$$
\xi_j(a_n)=\xi_j,
\quad
Q_t(\xi_j(t))=0,
$$
and
$$
\left.
\frac{d}{dt}|\xi_j(t)|^2
\right|_{t=a_n}
>0.
$$
Since there are only finitely many such zeros, there exists
$\varepsilon_0>0$ such that
\begin{equation}\label{eq:boundarymovesout}
|\xi_j(t)|>1,
\quad
a_n<t\leq a_n+\varepsilon_0,
\quad
j=1,\ldots,m.
\end{equation}

Choose pairwise disjoint closed discs $U_1,\ldots,U_m$, with $U_j$
centred at $\xi_j$, such that each $U_j$ contains no zero of
$Q_{a_n}$ other than $\xi_j$ and the boundary of $U_j$ contains no
zero of $Q_{a_n}$. By decreasing $\varepsilon_0$ if necessary, we may
assume that
$$
\xi_j(t)\in\operatorname{int} U_j,
\quad
a_n\leq t\leq a_n+\varepsilon_0,
\quad
j=1,\ldots,m.
$$
Let $\eta_1,\ldots,\eta_r$ be the distinct zeros of $Q_{a_n}$
outside the unit circle, and let $\mu_k$ denote the algebraic
multiplicity of $\eta_k$; set $r=0$ if no such zero exists. Unlike
the $\xi_j$, the exterior zeros $\eta_k$ need not be simple.

Choose pairwise disjoint closed discs $V_1,\ldots,V_r$, each centred
at the corresponding zero $\eta_k$, such that the discs are disjoint
from $U_1,\ldots,U_m$, lie entirely in the region $|w|>1$, and have
boundaries containing no zero of $Q_{a_n}$. Figure~\ref{fig:zero-neighbourhoods} illustrates the distinction
between the two families of neighbourhoods.

The points $\xi_1,\ldots,\xi_m,\eta_1,\ldots,\eta_r$ are all the
distinct zeros of $Q_{a_n}$. Since this polynomial has degree $n$,
at least one of the two families of discs is non-empty.

Let
$$
\Gamma
:=
\bigcup_{j=1}^m\partial U_j
\;\cup\;
\bigcup_{k=1}^r\partial V_k.
$$
This is a non-empty compact set, and $Q_{a_n}$ has no zero on
$\Gamma$. Hence
$$
\delta
:=
\min_{w\in\Gamma}|Q_{a_n}(w)|
>0.
$$
Choose $\varepsilon>0$ such that
$\varepsilon\leq\varepsilon_0$ and $\varepsilon<\delta$. If
$a_n<t\leq a_n+\varepsilon$, then, for every $w\in\Gamma$,
$$
|Q_t(w)-Q_{a_n}(w)|
=
t-a_n
<
\delta
\leq
|Q_{a_n}(w)|.
$$
Rouch\'e's theorem therefore applies on each of the contours
$\partial U_j$ and $\partial V_k$, with the same choice of
$\varepsilon$. It shows that $Q_t$ has exactly one zero in
$\operatorname{int}U_j$, namely $\xi_j(t)$, and exactly $\mu_k$
zeros, counted with multiplicity, in $\operatorname{int}V_k$.

By~\eqref{eq:boundarymovesout} and the choice of the discs $V_k$, all
these zeros lie strictly outside the unit circle. Since each $\xi_j$
is simple and each $\eta_k$ has algebraic multiplicity $\mu_k$,
these neighbourhoods account for all
$m+\sum_{k=1}^r\mu_k=n$ zeros of $Q_t$, counted with multiplicity.
Thus every zero of $Q_t$ has modulus greater than one whenever
$a_n<t\leq a_n+\varepsilon$. Put
$t_1:=a_n+\varepsilon$.

\begin{figure}[htbp]
\centering
\begin{tikzpicture}[
  x=1cm,
  y=1cm,
  line cap=round,
  line join=round
]
  \tikzset{
    unitcircle/.style={
      draw=black!78,
      line width=0.82pt
    },
    Udisc/.style={
      draw=black!55,
      fill=black!3,
      line width=0.62pt
    },
    Vdisc/.style={
      draw=black!82,
      fill=black!14,
      line width=0.72pt
    },
    root/.style={
      fill=black!95
    },
    rootring/.style={
      draw=black!90,
      line width=0.45pt
    },
    legend/.style={
      anchor=north west,
      align=left,
      text width=4.45cm,
      font=\footnotesize,
      inner sep=0pt
    }
  }

  \coordinate (c) at (-0.75,0);
  \path[fill=black!1] (c) circle (2);

  \path (c) ++(35:2) coordinate (uone);
  \path (c) ++(155:2) coordinate (utwo);
  \path (c) ++(250:2) coordinate (uthree);

  \draw[Udisc] (uone) circle (0.38);
  \draw[Udisc] (utwo) circle (0.36);
  \draw[Udisc] (uthree) circle (0.34);

  \coordinate (vone) at (2.45,0.85);
  \coordinate (vtwo) at (2.30,-1.08);

  \draw[Vdisc] (vone) circle (0.56);
  \draw[Vdisc] (vtwo) circle (0.50);

  \draw[unitcircle] (c) circle (2);

  \fill[root] (uone) circle (1.55pt);
  \fill[root] (utwo) circle (1.55pt);
  \fill[root] (uthree) circle (1.55pt);

  \fill[root] (vone) circle (1.55pt);
  \draw[rootring] (vone) circle (3.8pt);
  \fill[root] (vtwo) circle (1.55pt);
  \draw[rootring] (vtwo) circle (3.8pt);

  \draw[black!22,line width=0.4pt]
    (-5.45,-2.58) -- (5.30,-2.58);

  \draw[Udisc] (-5.25,-3.08) circle (0.20);
  \draw[unitcircle] (-5.55,-3.08) -- (-4.95,-3.08);
  \fill[root] (-5.25,-3.08) circle (1.35pt);

  \node[legend] at (-4.82,-2.84)
    {$\boldsymbol{U_j}$. Centre $\xi_j\in\partial\mathbb D$, with
     $\operatorname{mult}_{Q_{a_n}}(\xi_j)=1$. For $t>a_n$
     sufficiently close to $a_n$, $Q_t$ has one zero in
     $\operatorname{int}U_j$.};

  \draw[Vdisc] (0.35,-3.08) circle (0.20);
  \fill[root] (0.35,-3.08) circle (1.35pt);
  \draw[rootring] (0.35,-3.08) circle (3.3pt);

  \node[legend] at (0.77,-2.84)
    {$\boldsymbol{V_k}$. Centre $\eta_k$, with $|\eta_k|>1$ and
     $\operatorname{mult}_{Q_{a_n}}(\eta_k)=\mu_k\geq1$. For
     $t>a_n$ sufficiently close to $a_n$, $Q_t$ has $\mu_k$ zeros in
     $\operatorname{int}V_k$, counted with multiplicity.};
\end{tikzpicture}

\caption[Local zero neighbourhoods for $Q_{a_n}$]
{The two local configurations at $t=a_n$. Boundary zeros are simple,
whereas no simplicity assumption is imposed on exterior zeros. Ringed
markers record this distinction but do not encode a particular value
of $\mu_k$. The drawing is not to scale.}
\label{fig:zero-neighbourhoods}
\end{figure}
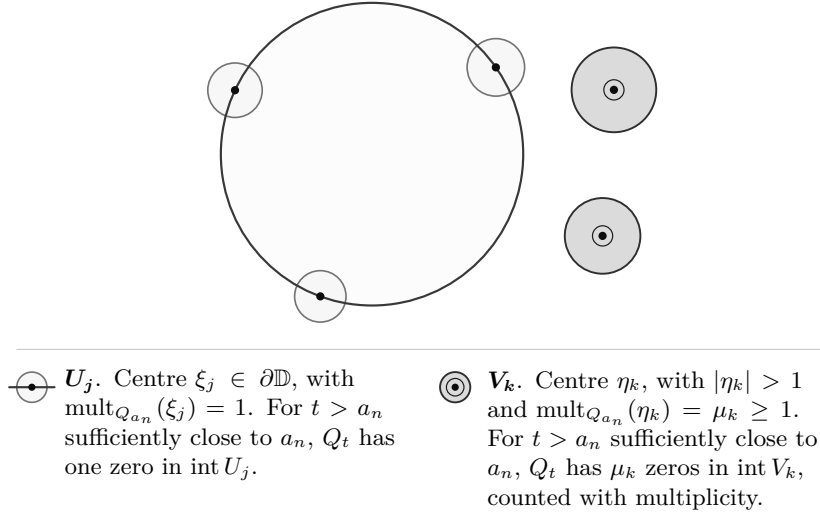

\subsubsection*{Exclusion of a first return}
Assume, for a contradiction, that $Q_t$ has a zero in
$\overline{\mathbb D}$ for some $t>t_1$, and define
$$
\mathcal B
:=
\left\{
t\geq t_1:
Q_t(w)=0
\text{ for some }w\in\overline{\mathbb D}
\right\}.
$$
This set is non-empty by assumption.

We first note that $\mathcal B$ is closed. Suppose that
$(t_\ell)$ is a sequence in $\mathcal B$ converging to $t$. For
each $\ell$, choose $w_\ell\in\overline{\mathbb D}$ with
$Q_{t_\ell}(w_\ell)=0$. By compactness, after passing to a subsequence,
we may suppose that $w_\ell\to w$ for some
$w\in\overline{\mathbb D}$. Since $t_\ell\to t$, the
continuity of $(s,z)\mapsto Q_s(z)$ gives
$$
Q_t(w)
=
\lim_{\ell\to\infty}Q_{t_\ell}(w_\ell)
=
0.
$$
Moreover, $t_\ell\geq t_1$ for every $\ell$, and hence $t\geq t_1$.
Thus $t\in\mathcal B$, as required. Since $\mathcal B$ is non-empty and bounded below, put
$$
T:=\inf\mathcal B.
$$

For each positive integer $\ell$, one may choose
$s_\ell\in\mathcal B$ with $T\leq s_\ell<T+1/\ell$. Hence
$s_\ell\to T$,
and closedness gives $T\in\mathcal B$; in particular,
$T=\min\mathcal B$. On the other hand, the choice of $t_1$ ensures
that every zero of $Q_{t_1}$ has modulus greater than $1$. Thus
$t_1\notin\mathcal B$, and consequently $T>t_1$.

Because $T\in\mathcal B$, the polynomial $Q_T$ has a zero $w_T$
with $|w_T|\leq1$. We claim that $|w_T|=1$. Were
$|w_T|<1$, we could choose a closed disc $\Delta$, centred at
$w_T$ and contained in $\mathbb D$, such that $Q_T$ has no zero
on $\partial\Delta$. Since
$$
Q_t(w)-Q_T(w)=t-T,
$$
Rouch\'e's theorem shows that, for $t<T$ sufficiently close to
$T$, the polynomials $Q_t$ and $Q_T$ have the same number of
zeros, counted with multiplicity, in $\operatorname{int}\Delta$.
This number is positive because $w_T\in\operatorname{int}\Delta$. We
may also require $t_1<t<T$, which would place $t$ in
$\mathcal B$ and contradict the minimality of $T$. The claim
follows.

By the preceding local analysis, the simple zero $w_T$ lies on a
unique local $C^1$ branch $w(t)$ such that
$$
\begin{aligned}
w(T)&=w_T,
\quad
Q_t(w(t))=0,
\\[7pt]
\left.
\frac{d}{dt}|w(t)|^2
\right|_{t=T}
&>0.
\end{aligned}
$$
Hence, as $t\longrightarrow T$,
$$
|w(t)|^2
=
1+
\left.
\frac{d}{dt}|w(t)|^2
\right|_{t=T}
(t-T)
+
o(|t-T|).
$$
It follows that $|w(t)|<1$ whenever $t<T$ is sufficiently close to
$T$. Since $T>t_1$, such a $t$ may be chosen in the interval
$(t_1,T)$. Once again $t\in\mathcal B$, contrary to the minimality
of $T$. This contradiction shows that $\mathcal B$ is empty. Together with the
local conclusion obtained for $a_n<t\leq t_1$, it follows that every
zero $w$ of $Q_t$ satisfies $|w|>1$ whenever $t>a_n$. The
reciprocal correspondence then places every zero of $P_t$ in
$\mathbb D$. Finally, setting $t=a_n+d$ gives $P_t=P_d$ and
completes the proof.
\end{proof}

\section{Proof of Lemma~\ref{lem:trig}}\label{sec:trig-lemma}

\begin{proof}
Set
$$
b_{n+1}:=0,
\quad
\delta_k:=b_k-b_{k+1},
\quad
k=1,\ldots,n.
$$
By the monotonicity of the sequence $(b_j)$,
$$
\delta_k\geq0,
\quad
k=1,\ldots,n.
$$
Moreover, telescoping gives
\begin{equation}\label{eq:btelescope}
b_j=\sum_{k=j}^n\delta_k,
\quad
j=1,\ldots,n,
\end{equation}
and, in particular,
\begin{equation}\label{eq:b1delta}
b_1=\sum_{k=1}^n\delta_k.
\end{equation}

Using~\eqref{eq:btelescope} and interchanging two finite sums, the
hypothesis~\eqref{eq:imagcondition} becomes
\begin{equation}\label{eq:imagdelta}
\sum_{k=1}^n\delta_k
\sum_{j=1}^k\sin(j\theta)=0.
\end{equation}

Let
$$
\Lambda:=(n-1)b_1+
\sum_{j=1}^n(n-j)b_j\cos(j\theta).
$$
It follows from~\eqref{eq:btelescope} and~\eqref{eq:b1delta} that
\begin{equation}\label{eq:Ldelta}
\Lambda
=
\sum_{k=1}^n\delta_k
\left[
(n-1)+
\sum_{j=1}^k(n-j)\cos(j\theta)
\right].
\end{equation}

Put $x:=\theta/2$. Then $0<x<\pi$, so $\sin x>0$ and
$\cot x$ is well defined. Writing $\theta=2x$ and adding
$(n-1)\cot x$ times equation~\eqref{eq:imagdelta} to the right-hand
side of~\eqref{eq:Ldelta} gives
\begin{equation}\label{eq:LE}
\Lambda=\sum_{k=1}^n\delta_k E_{n,k}(x),
\end{equation}
where
\begin{equation}\label{eq:Enkdef}
\begin{aligned}
E_{n,k}(x)
:={}&
(n-1)+
\sum_{j=1}^k(n-j)\cos(2jx)
\\[7pt]
&+
(n-1)\cot x
\sum_{j=1}^k\sin(2jx).
\end{aligned}
\end{equation}
Since $\delta_k\geq0$, it is enough to prove that
$$
E_{n,k}(x)\geq0,
\quad
k=1,\ldots,n.
$$

For $r=0,1,\ldots$, define
$$
U_r(x):=\frac{\sin((r+1)x)}{\sin x}.
$$
Equivalently, $U_r(x)$ is obtained by evaluating the degree-$r$
Chebyshev polynomial of the second kind at $\cos x$.

We first observe that
$$
\cos(2jx)+\cot x\,\sin(2jx)
=
\frac{\sin((2j+1)x)}{\sin x}.
$$
Together with the finite trigonometric sum
$$
\sum_{j=0}^k\sin((2j+1)x)
=
\frac{\sin^2((k+1)x)}{\sin x},
$$
this gives
\begin{equation}\label{eq:Ukidentity}
1+
\sum_{j=1}^k
\left[
\cos(2jx)+\cot x\,\sin(2jx)
\right]
=
U_k(x)^2.
\end{equation}

Writing $n-j=(n-1)-(j-1)$ in~\eqref{eq:Enkdef} and
using~\eqref{eq:Ukidentity}, we obtain
$$
E_{n,k}(x)
=
(n-1)U_k(x)^2
-
\sum_{j=1}^k(j-1)\cos(2jx).
$$
Set
\begin{equation}\label{eq:Dkdef}
D_k(x)
:=
(k-1)U_k(x)^2
-
\sum_{j=1}^k(j-1)\cos(2jx).
\end{equation}
Then
\begin{equation}\label{eq:EnkDk}
E_{n,k}(x)
=
(n-k)U_k(x)^2+D_k(x).
\end{equation}
It remains to show that
$$
D_k(x)\geq0,
\quad
k=1,\ldots,n.
$$

For $r=0,1,\ldots$, we have the finite Fourier expansion
\begin{equation}\label{eq:finitefourier}
U_r(x)^2
=
(r+1)+
2\sum_{j=1}^r(r+1-j)\cos(2jx).
\end{equation}
Indeed,
$$
U_r(x)
=
\sum_{\ell=0}^r e^{i(r-2\ell)x}.
$$
Since this sum is real, we have
$$
\begin{aligned}
U_r(x)^2
&=
\left|
\sum_{\ell=0}^r e^{i(r-2\ell)x}
\right|^2
\\[7pt]
&=
\sum_{\ell,m=0}^r e^{2i(m-\ell)x}
\\[7pt]
&=
(r+1)+
2\sum_{j=1}^r(r+1-j)\cos(2jx),
\end{aligned}
$$
which proves~\eqref{eq:finitefourier}.

Applying~\eqref{eq:finitefourier} with $r=k$ and $r=k-1$, a direct
calculation gives
\begin{equation}\label{eq:DkFourier}
2D_k(x)
=
kU_{k-1}(x)^2
+
(k-1)U_k(x)^2
-1.
\end{equation}

For $k=1$, definition~\eqref{eq:Dkdef} immediately gives
$$
D_1(x)=0.
$$
For $k=2$, the elementary formulae for $U_1$ and $U_2$, together
with~\eqref{eq:DkFourier}, give
$$
\begin{aligned}
U_1(x)&=2\cos x,
\quad
U_2(x)=4\cos^2x-1,
\\[7pt]
2D_2(x)&=2U_1(x)^2+U_2(x)^2-1
=16\cos^4x,
\\[7pt]
D_2(x)&=8\cos^4x\geq0.
\end{aligned}
$$

Suppose now that $k\geq3$. We use the elementary identity
\begin{equation}\label{eq:Ukquadratic}
U_k(x)^2+U_{k-1}(x)^2
-
2\cos x\,U_k(x)U_{k-1}(x)
=
1,
\end{equation}
which follows directly from the sine addition formula.
Combining~\eqref{eq:Ukquadratic} with~\eqref{eq:DkFourier}, we obtain
$$
\begin{aligned}
2D_k(x)
&=
(k-2)U_k(x)^2
+
(k-1)U_{k-1}(x)^2+
2\cos x\,U_k(x)U_{k-1}(x)
\\[7pt]
&\geq
(k-3)U_k(x)^2
+
(k-2)U_{k-1}(x)^2
\\[7pt]
&
\geq0.
\end{aligned}
$$
Here the first inequality uses $|\cos x|\leq1$ and
$2|uv|\leq u^2+v^2$.
Thus $D_k(x)\geq0$ for every $1\leq k\leq n$. Since
$n-k\geq0$, equation~\eqref{eq:EnkDk} then gives
$E_{n,k}(x)\geq0$ throughout the same range.
Finally,~\eqref{eq:LE} and $\delta_k\geq0$ give
$$
\Lambda
=
\sum_{k=1}^n\delta_k E_{n,k}(x)
\geq0.
$$
This proves~\eqref{eq:lemmaineq} and hence the lemma.
\end{proof}

\section{Application to Brown methods}

For each positive integer $L$, define
$$
\begin{aligned}
K_L^\ast
&:=
\min_{1\leq j\leq L}
\left\{
\frac{(j+1)^{L+1}}{2j^L}+j
\right\},
\\[7pt]
K_L
&:=
\min_{1\leq j\leq L}
\left\{
\frac{(j+1)^{L+1}}{j^L}+j
\right\}.
\end{aligned}
$$

\begin{corollary}[Jeltsch's conjecture for a subclass of Brown methods]
\label{cor:brown}
Let $K$ and $L$ be positive integers such that
$$
L\geq10,
\quad
K\leq K_L.
$$
Then every zero-stable Brown $(K,L)$ method is stiffly stable.
\end{corollary}

\begin{proof}
Let $\alpha_0,\ldots,\alpha_K\in\mathbb R$ denote the characteristic
coefficients of the method in Meneguette's notation. In his root-locus
analysis, it is sufficient for $A_0$-stability to prove that, for
every $\gamma>0$, all the zeros of
$$
p_\gamma(z)
=
(\alpha_K+\gamma)z^K
+
\alpha_{K-1}z^{K-1}
+\cdots+
\alpha_0
$$
lie in $\mathbb D$~\cite[Theorem~2.13, Remark~2.16 and
pp.~81--82, 105--107]{M87}.

Set
$$
\widetilde p_\gamma(z)
:=
(-1)^Kp_\gamma(-z).
$$
Under this transformation each zero is replaced by its negative, so
its modulus is unchanged. For $0\leq j\leq K-1$, set
$$
c_j:=(-1)^{K+j}\alpha_j,
\quad
c_K:=\alpha_K.
$$
In Meneguette's positive-coefficient formulation, the sign formula for
the characteristic coefficients gives $c_j>0$ for
$0\leq j\leq K$~\cite[(2.16), p.~65; p.~105]{M87}.
Moreover,
$$
\widetilde p_\gamma(z)
=
c_0+c_1z+\cdots+c_{K-1}z^{K-1}
+(c_K+\gamma)z^K.
$$
Meneguette's coefficient estimates show that the sequence
$c_0,\ldots,c_{K-1}$ is non-decreasing whenever
$K\leq K_L$~\cite[Theorem~2.9, pp.~66--67]{M87}. For $L\geq10$,
they also give
$$
Kc_K>(K-1)c_{K-1}
$$
throughout the same range of $K$~\cite[pp.~106--107]{M87}.

If the Brown method is zero-stable, then $p_0$ satisfies the root
condition~\cite[Definition~2.4 and pp.~46--47]{M87}. Hence all the zeros of
$\widetilde p_0$ lie in $\overline{\mathbb D}$.
Theorem~\ref{thm:main}, applied with $n=K$,
$P=\widetilde p_0$, and $d=\gamma$, places every zero of
$\widetilde p_\gamma$ in $\mathbb D$. Since the transformation
preserves moduli, the same is true of $p_\gamma$. Thus the method is
$A_0$-stable.

Meneguette also proves that the method is strongly stable for
$K\leq K_L$~\cite[Theorem~2.12, pp.~70--72]{M87}. Finally, stiff
stability is equivalent to the conjunction of $A_0$-stability and
strong stability~\cite[Theorem~2.16, pp.~96--97]{M87}. The method is
therefore stiffly stable.
\end{proof}

For comparison, the coefficient-ratio argument of~\cite{BMCM12}
establishes the corresponding implication when $K\leq K_L^\ast$.
Thus Corollary~\ref{cor:brown} extends the previously verified bound
from $K_L^\ast$ to $K_L$ when $L\geq10$.

\section*{Acknowledgements}
The author is deeply grateful to V.~Botta for drawing his attention to
this problem, for generously sharing the relevant sources---among them
M.~Meneguette's doctoral thesis---and for many stimulating
discussions. He also thanks M.~Meneguette for agreeing to read the
manuscript and comment on the proof. The author acknowledges financial support from the Centre for
Mathematics of the University of Coimbra (CMUC), funded by the
Portuguese Foundation for Science and Technology (FCT), under the
projects UID/00324/2025
(\url{https://doi.org/10.54499/UID/00324/2025}) and
UID/PRR/00324/2025. The author also acknowledges financial support
from the FCT under the grant
\url{https://doi.org/10.54499/2022.00143.CEECIND/CP1714/CT0002}.

\end{document}